\documentclass[11pt,a4paper,reqno]{article}

\usepackage{amssymb}
\usepackage{latexsym}
\usepackage{amsmath}
\usepackage{graphicx}
\usepackage{amsthm}
\usepackage{empheq}
\usepackage{multirow}
\usepackage{diagbox}
\usepackage{bm}
\usepackage{booktabs}
\usepackage[dvipsnames]{xcolor}
\usepackage{pagecolor}
\usepackage{subcaption}
\usepackage{tikz,lipsum,lmodern}
\usepackage{enumitem}
\usepackage{calligra}
\usepackage{mathrsfs}
\usepackage[margin=1in]{geometry}
\usepackage{authblk}
\usepackage{enumitem}
\setitemize{itemsep=0em}
\setenumerate{itemsep=0em}
\usepackage{hyperref}
\hypersetup{
  colorlinks   = true, 
  urlcolor     = blue, 
  linkcolor    = purple, 
  citecolor   = ForestGreen 
}

\usepackage[margin=1in]{geometry}

\numberwithin{equation}{section}

\theoremstyle{definition}
\newtheorem{theorem}{Theorem}[section]
\newtheorem{corollary}[theorem]{Corollary}
\newtheorem{proposition}[theorem]{Proposition}
\newtheorem{definition}[theorem]{Definition}
\newtheorem{example}[theorem]{Example}

\newtheorem{remark}[theorem]{Remark}

\newcommand{\numberset}{\mathbb}

\newcommand{\F}{\numberset{F}}

\newcommand{\fq}{\F_q}
\newcommand{\Fq}{\F_q}

\newcommand{\fqm}{\F_{q^m}}

\newcommand{\Tr}{\textnormal{Tr}}

\newcommand{\rk}{\textnormal{rk}}

\renewcommand{\longrightarrow}{\to}

\newcommand{\GL}{\textnormal{GL}_m(q)}

\def\sym{\textup{sym}}
\def\ss{\textup{ss}}

\def\cov{\mathrel{<\kern-.6em\raise.015ex\hbox{$\cdot$}}}
\def\<{\left<}
\def\>{\right>}

\newcommand{\Sym}{{\textnormal{Sym}}}

\newcommand*{\myproofname}{Proof of the claim}

\allowdisplaybreaks

\title{\textbf{Symmetric Tensor Decompositions over Finite Fields}}

\usepackage{setspace}
\author[1]{Giuseppe Cotardo}
\affil[1]{Virginia Tech, Blacksburg, U.S.A.}

\author[2]{Ferdinando Zullo\thanks{F. Z. was partially supported by the project COMBINE from University of Campania and by the Italian National Group for Algebraic and Geometric Structures and their Applications (GNSAGA - INdAM).}}
\affil[2]{Universit\`a degli Studi della Campania ``Luigi Vanvitelli'', Caserta, Italy.}

\date{}

\usepackage{setspace}
\begin{document}
\maketitle
	
\abstract{
We study the symmetric tensor rank of multiplication over finite field
extensions using linearized polynomials. Via field trace, symmetric
linearized polynomials are identified with symmetric bilinear forms and
symmetric matrices, allowing symmetric tensor decompositions to be
reformulated as spanning problems by rank-one symmetric linearized
polynomials. We translate these spanning conditions into explicit linear
systems over finite fields and use the Frobenius automorphism to obtain
computationally effective criteria. As applications, we recover known values
of the symmetric bilinear complexity for small extension degrees and obtain
explicit symmetric decompositions for several parameters. We also introduce
the symmetric tensor-rank of a symmetric rank-metric code and show that, for
the natural one-dimensional Gabidulin code associated with finite field
multiplication, this invariant coincides with the symmetric tensor rank of
the multiplication map.
}

\medskip

\textbf{Keywords:}{ symmetric tensor rank; finite field multiplication; bilinear complexity; linearized polynomials; tensor decompositions; rank-metric codes; Gabidulin codes.}\\

\textbf{MSC2020:}{ 03D15; 15A69; 12E20; 94B05}

\section*{Introduction}

The study of tensor rank and tensor decompositions is a central topic in
algebraic complexity theory, where tensor rank measures the complexity of
bilinear maps and, in particular, the number of essential scalar
multiplications needed to evaluate them. A classical problem in this area is
the bilinear complexity of multiplication in finite field extensions. More
precisely, given the multiplication map

\[
\begin{array}{cccc}
M_{q^m}:&\mathbb F_{q^m}\times \mathbb F_{q^m}&
\longrightarrow &\mathbb F_{q^m}\\&(\alpha,\beta)&\longmapsto &\alpha\beta,
\end{array}
\]
one asks for the smallest number of bilinear products required to compute
\(\alpha\beta\) over the base field \(\mathbb F_q\). This number is the
tensor rank of \(M_{q^m}\), usually denoted by \(\mu_q(m)\), and is also
known as the bilinear complexity of multiplication in \(\mathbb F_{q^m}\).
The problem has a long history and has been studied through methods from
algebraic complexity theory, algebraic geometry, and coding theory; see
\cite{ballet2021tensor,burgisser2013algebraic} and the references therein.

Since multiplication in \(\mathbb F_{q^m}\) is commutative, the tensor
associated with \(M_{q^m}\) has additional symmetry. It is therefore natural
to consider symmetric decompositions, in which the two linear forms appearing
in each rank-one summand are required to coincide. The minimum length of such
a decomposition is the symmetric tensor rank of the multiplication map,
denoted by \(\mu_q^{\mathrm{sym}}(m)\). This invariant is generally at least
\(\mu_q(m)\), and it captures the complexity of algorithms for finite field
multiplication under a natural symmetry constraint. Symmetric bilinear
algorithms over finite fields were already studied in
\cite{seroussi1984symmetric}. More broadly, symmetric multiplication
algorithms have been investigated through algebraic-geometric constructions,
including the Chudnovsky--Chudnovsky method and its variants
\cite{ballet2021tensor,chaumine2006bilinear,chudnovsky1988algebraic,shokrollahi1992optimal}.

Despite substantial progress, determining exact values of
\(\mu_q^{\mathrm{sym}}(m)\) remains difficult. The lower bound $\mu_q(m)\geq 2m-1$ also applies to \(\mu_q^{\mathrm{sym}}(m)\), and equality holds for both
ordinary and symmetric tensor rank when \(q\geq 2m-2\); see
\cite[Theorem~2.2]{ballet2021tensor}. Further exact values are known in
special ranges, for instance through elliptic-curve constructions and
related algebraic-geometric methods
\cite{chaumine2006bilinear,chudnovsky1988algebraic, shokrollahi1992optimal}.
Outside these regimes, only a limited number of exact values are known, and
many cases are understood only through bounds.

The goal of this paper is to develop a constructive approach to symmetric
tensor decompositions of \(M_{q^m}\) using linearized polynomials. Linearized
polynomials, originally introduced by Ore, provide a natural model for
\(\mathbb F_q\)-linear maps on finite field extensions
\cite{lidl1997finite,ore1933special}. Via the field trace, they are closely
related to bilinear forms and matrix representations. We exploit this
correspondence in the symmetric setting: symmetric linearized polynomials
correspond to symmetric bilinear forms, and rank-one symmetric linearized
polynomials admit trace representations of the form
\[
\alpha \operatorname{Tr}_{q^m/q}(\beta x),
\]
with a suitable relation between \(\alpha\) and \(\beta\). This allows us to
translate the problem of finding symmetric tensor decompositions into an
explicit linear-algebraic problem over finite fields.

More precisely, we show that the existence of a symmetric tensor
decomposition of length~\(R\) can be expressed as a spanning condition among
rank-one symmetric linearized polynomials. After choosing a generator of
\(\mathbb F_{q^m}\) over \(\mathbb F_q\), this spanning condition becomes a
system of explicit linear equations over \(\mathbb F_q\). We then use the
Frobenius automorphism to enlarge these systems in a way that is particularly
well suited for computation. This gives a practical criterion for verifying
candidate decompositions and provides a systematic method for producing
explicit examples.

As applications, we recover known values of \(\mu_q^{\mathrm{sym}}(m)\) for
small extension degrees and provide explicit symmetric tensor decompositions
in several cases. In particular, we focus on the cases \(m\in\{2,3,4\}\).
The cases \(m=2\) and \(m=3\) illustrate how the linear systems specialize
to concrete determinantal conditions, while the case \(m=4\) shows how the
method can be used computationally to obtain explicit decompositions over
small fields.

We also reinterpret the construction in the language of symmetric
rank-metric codes. Rank-metric codes were introduced by Delsarte in terms of
bilinear forms and later developed in the extension-field setting by
Gabidulin, giving rise to the family now known as Gabidulin codes
\cite{delsarte1978bilinear,gabidulin1985theory}. The first-slice space of
the multiplication tensor is equivalent to the one-dimensional Gabidulin code
\[
\mathcal{G}_1(m,q)=\langle x\rangle_{\mathbb F_{q^m}},
\]
which, under the field trace, can be viewed as a space of symmetric
linearized polynomials. Motivated by this observation and by the
matrix-slice characterization of tensor rank, we introduce the symmetric
tensor rank of a symmetric rank-metric code as the minimum number of
rank-one symmetric matrices whose \(\mathbb F_q\)-span contains the code.
With this definition, the symmetric tensor rank of \(M_{q^m}\) agrees with
the symmetric tensor rank of the natural one-dimensional symmetric
rank-metric code associated with \(\mathcal{G}_1(m,q)\). This perspective is
closely related to the study of tensor rank for rank-metric codes in
\cite{byrne2023tensor,byrne2019tensor}, and extends it to the symmetric
setting.

The paper is organized as follows. In Section~\ref{sec:1}, we recall the
necessary background on linearized polynomials, symmetric bilinear forms,
and symmetric matrices, and we describe the rank-preserving correspondence
among these objects. In Section~\ref{sec:2}, we review tensor rank and
symmetric tensor rank for finite field multiplication and recall the
connection between the first-slice space of \(M_{q^m}\) and the
one-dimensional Gabidulin code. Section~\ref{sec:symdec} contains the main
linearized-polynomial framework and the Frobenius-stable linear systems used
to construct decompositions. In Section~\ref{sec:4}, we apply the method to obtain explicit decompositions for \(m\in\{2,3,4\}\). Finally, in
Section~\ref{sec:rankmetriccodes}, we introduce the symmetric tensor rank of
symmetric rank-metric codes and discuss the resulting coding-theoretic
interpretation.

\section{Symmetric bilinear forms, symmetric matrices and linearized polynomials}\label{sec:1}

In this section, we recall the algebraic structures used throughout the paper:
linearized polynomials, symmetric bilinear forms, and matrix representations
over finite fields. These objects provide equivalent viewpoints for studying
\(\mathbb F_q\)-linear maps on \(\mathbb F_{q^m}\).

We begin by introducing the algebra of linearized polynomials over
\(\mathbb F_{q^m}\), which is in one-to-one correspondence with
\(\mathbb F_q\)-linear endomorphisms of \(\mathbb F_{q^m}\). We then focus on
the subclass of symmetric linearized polynomials, defined via the trace
bilinear form, and describe their connection with symmetric bilinear forms and symmetric matrices. This correspondence will play a central role in the sequel. It allows us to
translate questions about symmetric tensor rank into questions about
linearized polynomials and rank properties of the associated matrices. In
particular, since these representations preserve rank, we will freely switch
between the polynomial, bilinear-form, and matrix viewpoints depending on the
context.

\subsection{Symmetric linearized polynomials}

Let \(\mathcal L_{m,q}\) denote the \(\mathbb F_q\)-algebra of \(q\)-\textbf{polynomials} (or \textbf{linearized polynomial})
over \(\mathbb F_{q^m}\), considered modulo \(x^{q^m}-x\). Equivalently,
\[
\mathcal L_{m,q}
=
\left\{
\sum_{i=0}^{m-1} a_i x^{q^i}: a_i\in \mathbb F_{q^m}
\right\}.
\]
It is well known that \(\mathcal L_{m,q}\) is isomorphic to the algebra of
\(\mathbb F_q\)-linear endomorphisms of \(\mathbb F_{q^m}\), and hence, after
fixing an \(\mathbb F_q\)-basis, to \(\mathbb F_q^{m\times m}\); see~\cite{WU201379}.
We identify \(f(x)\in L_{m,q}\) with the \(\mathbb F_q\)-linear map
\(x\mapsto f(x)\) on \(\mathbb F_{q^m}\). Henceforth, we will talk about the kernel of a linearized polynomial $f(x) \in \mathcal{L}_{m,q}$, meaning the kernel of the $\fq$-linear map $f(x)$ over $\mathbb{F}_{q^m}$.

Consider the non-degenerate symmetric bilinear form of $\F_{q^m}$ over $\F_q$ (see next section for details) defined for every $x,y \in \F_{q^m}$ by
\begin{equation}\label{eq:bilform} \langle x,y\rangle= \mathrm{Tr}_{q^m/q}(xy). \end{equation}
The adjoint \(f^\top\) of the \(q\)-polynomial
\[
f(x)=\sum_{i=0}^{m-1}a_i x^{q^i}
\]
with respect to \(\langle\cdot,\cdot\rangle\) is the unique element of
\(\mathcal L_{m,q}\) satisfying
\[
\operatorname{Tr}_{q^m/q}(y f(z))
=
\operatorname{Tr}_{q^m/q}(z f^\top(y))
\]
for all \(y,z\in\mathbb F_{q^m}\). Explicitly, we have
\[
f^\top(x)
=
a_0x+\sum_{i=1}^{m-1} a_i^{q^{m-i}}x^{q^{m-i}}.
\]
Linearized polynomials that coincide with their adjoints are called \textbf{symmetric linearized polynomials} for their connection with symmetric matrices; see~\cite{longobardi2020automorphism} and the next subsection. We will use the following notation for the vector space of symmetric linearized
polynomials.
\[
\mathcal S_q(m)=
\left\{
\sum_{i=0}^{m-1} c_i x^{q^i}
:
c_{m-i}=c_i^{q^{m-i}}
\text{ for } i\in\{0,\ldots,m-1\}
\right\}
\subseteq \mathcal L_{m,q},
\]
where the indices are read modulo \(m\). The following result provides a characterization of rank-one symmetric linearized polynomials.

\begin{proposition}\label{prop:rank1sym}
An element \(f(x)\in \mathcal S_q(m)\) has rank one if and only if
\[
f(x)=\alpha\operatorname{Tr}_{q^m/q}(\beta x)
\]
for some \(\alpha,\beta\in\mathbb F_{q^m}^*\) with \(\alpha/\beta\in\mathbb F_q\).
\end{proposition}
\begin{proof}
    By \cite[Theorem 2.24]{lidl1997finite}, a linearized polynomial $f$ has rank one if and only if there exist $\alpha, \beta \in \fqm^*$ such that $f(x)=\alpha \mathrm{Tr}_{q^m/q}(\beta x)$. It remains to determine when such a rank-one map is symmetric. The adjoint of $f(x)=\alpha\operatorname{Tr}_{q^m/q}(\beta x)$ is
\[
f^\top(x)=\beta\operatorname{Tr}_{q^m/q}(\alpha x).
\]
Thus \(f=f^\top\) if and only if $\alpha\operatorname{Tr}_{q^m/q}(\beta x)
=
\beta\operatorname{Tr}_{q^m/q}(\alpha x)$ for every \(x\in\mathbb F_{q^m}\), which holds if and only if
\(\alpha/\beta\in\mathbb F_q\).
\end{proof}

\subsection{Symmetric bilinear forms}

Let $V$ be an $m$-dimensional vector space over $\mathbb{F}_q$. A \textbf{bilinear form} on $V$ is a map $B : V \times V \to \mathbb{F}_q$ that is linear in each argument. If $B$ satisfies $B(u,v)=B(v,u)$ for all $u,v \in V$ then we say that $B$ is \textbf{symmetric}. After fixing a basis of \(V\), every bilinear form is represented by a matrix
\(A\in\mathbb F_q^{m\times m}\) via
\[
B(u,v)=u^TAv.
\]
The form \(B\) is symmetric if and only if \(A\) is symmetric.
Two symmetric bilinear forms are said to be \textbf{equivalent} if their associated matrices are congruent, that is if $A' = P^{T}AP$ for some $P \in \mathrm{GL}_m(\mathbb{F}_q)$. 
The classification of symmetric bilinear forms over finite fields depends on
whether \(\operatorname{char}(\mathbb F_q)=2\) and is closely related to
invariants such as rank and discriminant. In this paper, however, we will only
use the rank of a symmetric bilinear form, namely the rank of any matrix
representing it.
We denote by \(\mathrm{Sym}^2(V)\) the vector space of symmetric bilinear forms on \(V\).
The rank function induces the rank metric on \(\mathrm{Sym}^2(V)\), that is
\[
d(B_1,B_2)=\operatorname{rk}(B_1-B_2).
\]

We now specialize to the case \(V=\mathbb F_{q^m}\), viewed as an
\(m\)-dimensional vector space over \(\mathbb F_q\). In this setting, the
nondegenerate trace function allows one to represent bilinear forms by
linearized polynomials; see \cite[Lemma 13]{schmidt2010symmetric}.

\begin{proposition}\label{prop:bilform}
For each bilinear form \(B:\mathbb F_{q^m}\times \mathbb F_{q^m}\to \mathbb F_q\),
there exists a unique~\(f\in\mathcal L_{m,q}\) such that
\[
B(x,y)=\operatorname{Tr}_{q^m/q}(f(x)y)
\]
for every \(x,y\in\mathbb F_{q^m}\).
\end{proposition}

The next result shows that, under this correspondence, symmetric bilinear forms
are represented exactly by symmetric linearized polynomials; see \cite[Section 2.1]{longobardi2020automorphism}. We include a proof for completeness.

\begin{proposition}\label{prop:bilform2}
Let \(f\in\mathcal L_{m,q}\). The bilinear form
\[
B(x,y)=\operatorname{Tr}_{q^m/q}(f(x)y)
\]
is symmetric if and only if \(f\in\mathcal S_q(m)\).
\end{proposition}
\begin{proof}
The bilinear form \(B\) is symmetric if and only if $\operatorname{Tr}_{q^m/q}(f(x)y)
=
\operatorname{Tr}_{q^m/q}(f(y)x)$ for every \(x,y\in\mathbb F_{q^m}\). By the definition of the adjoint, $\operatorname{Tr}_{q^m/q}(f(x)y) =
\operatorname{Tr}_{q^m/q}(x f^\top(y))$.
Hence~\(B\) is symmetric if and only if $\operatorname{Tr}_{q^m/q}(x f^\top(y))
=
\operatorname{Tr}_{q^m/q}(f(y)x)$
for every \(x,y\in\mathbb F_{q^m}\). Since multiplication in
\(\mathbb F_{q^m}\) is commutative, this is equivalent to
\[
\operatorname{Tr}_{q^m/q}(x f^\top(y))
=
\operatorname{Tr}_{q^m/q}(x f(y))
\]
for every \(x,y\in\mathbb F_{q^m}\). Therefore, we have 
\[
\operatorname{Tr}_{q^m/q}\bigl(x(f^\top(y)-f(y))\bigr)=0
\]
for every \(x,y\in\mathbb F_{q^m}\). Since the trace form is nondegenerate,
we obtain \(f^\top(y)=f(y)\) for every \(y\in\mathbb F_{q^m}\), that is,
\(f=f^\top\). Hence \(B\) is symmetric if and only if
\(f\in\mathcal S_q(m)\).
\end{proof}

The previous proposition shows that the map
\[
\varphi:\mathcal S_q(m)\longrightarrow \operatorname{Sym}^2(V):f\longmapsto B_f,
\]
where
\[
B_f(x,y)=\operatorname{Tr}_{q^m/q}(f(x)y),
\]
is well defined. By Proposition~\ref{prop:bilform}, this map is a one-to-one correspondence
between symmetric linearized polynomials and symmetric bilinear forms. The next
result shows that this correspondence preserves rank.

\begin{theorem}
The map \(\varphi\) is an isometry between \(\mathcal S_q(m)\) and
\(\operatorname{Sym}^2(V)\) with respect to the rank metric.
\end{theorem}
\begin{proof}
By Propositions~\ref{prop:bilform} and~\ref{prop:bilform2}, the map \(\varphi\) is an
\(\mathbb F_q\)-linear isomorphism. It remains to show that it preserves rank. Let \(f\in\mathcal S_q(m)\). The radical of \(B_f\) is
\[
\operatorname{rad}(B_f)
=
\{x\in\mathbb F_{q^m}: B_f(x,y)=0 \text{ for all } y\in\mathbb F_{q^m}\}.
\]
Since $B_f(x,y)=\operatorname{Tr}_{q^m/q}(f(x)y)$, and the trace form is nondegenerate, we have $B_f(x,y)=0$ for all $y\in\mathbb F_{q^m}$ if and only if
$f(x)=0$. Hence $\operatorname{rad}(B_f)=\ker(f)$ and we get
\[
\operatorname{rk}(B_f)
=
m-\dim_{\mathbb F_q}\operatorname{rad}(B_f)
=
m-\dim_{\mathbb F_q}\ker(f)
=
\operatorname{rk}(f).
\]
Therefore \(\varphi\) preserves rank.
Finally, since \(\varphi\) is linear, for every \(f,g\in\mathcal S_q(m)\) we have
\[
d(B_f,B_g)
=
\operatorname{rk}(B_f-B_g)
=
\operatorname{rk}(B_{f-g})
=
\operatorname{rk}(f-g).
\]
Hence \(\varphi\) is an isometry with respect to the rank metric.
\end{proof}

\begin{remark}
    In view of the previous results, we will freely identify symmetric bilinear forms, symmetric linearized polynomials, and symmetric matrices once a basis is fixed. These identifications preserve rank and rank distance.
\end{remark}

Since we will repeatedly switch among these models, we illustrate the correspondence with an example.

\begin{example}
Let \(q=2\) and \(m=4\), and let \(\alpha\in \mathbb F_{2^4}\) be a root of \(x^4+x+1\). Consider the following ordered basis of $\F_{2^4}$
\[ \mathcal{B}=(1,\alpha,\alpha^2,\alpha^3). \]
Its trace-dual basis is 
\[ \mathcal{B}^*=(\alpha^{14},\alpha^2,\alpha,1). \]
The polynomials \(x,\alpha x,\alpha^2x\), and \(\alpha^3x\) are symmetric. The associated matrices, obtained by evaluating these polynomials on the trace-dual basis \(\mathcal{B}^*\) and expanding the images with respect to the basis \(\mathcal{B}\), are:
\[
\begin{pmatrix}
1 & 0 & 0 & 1 \\
0 & 0 & 1 & 0 \\
0 & 1 & 0 & 0 \\
1 & 0 & 0 & 0
\end{pmatrix},\quad
\begin{pmatrix}
1 & 0 & 0 & 0 \\
0 & 0 & 0 & 1 \\
0 & 0 & 1 & 0 \\
0 & 1 & 0 & 0
\end{pmatrix},\quad
\begin{pmatrix}
0 & 1 & 0 & 0 \\
1 & 1 & 0 & 0 \\
0 & 0 & 0 & 1 \\
0 & 0 & 1 & 0
\end{pmatrix},\quad
\begin{pmatrix}
0 & 0 & 1 & 0 \\
0 & 1 & 1 & 0 \\
1 & 1 & 0 & 0 \\
0 & 0 & 0 & 1
\end{pmatrix}.
\]
This illustrates explicitly how symmetric linearized polynomials are represented by symmetric matrices after choosing a basis and its trace-dual basis.
\end{example}

\section{Symmetric tensor rank of multiplication over finite fields}\label{sec:2}

In this section, we recall some definitions and results on the optimal bilinear complexity of the multiplication map 

\begin{equation*}
M_{q^m}:\fqm\times \fqm\rightarrow\fqm:(\alpha,\beta)\mapsto\alpha\beta.
\end{equation*}

A standard reference is~\cite{ballet2021tensor}. The bilinear complexity of an algorithm for \(M_{q^m}\) is the number of
non-scalar bilinear multiplications required to compute the product of two
elements of \(\mathbb F_{q^m}\). The \textbf{tensor rank}  (or \textbf{bilinear complexity}) of
\(M_{q^m}\), denoted by \(\mu_q(m)\), is the minimum length of such an algorithm.
Equivalently, it is the tensor rank of the tensor representing \(M_{q^m}\).  More explicitly, \(\mu_q(m)\) is the smallest integer \(R\) such that there exist
\(\phi_i,\psi_i\in \mathbb F_{q^m}^{\vee}\) and
\(w_i\in\mathbb F_{q^m}\), for \(i=1,\ldots,R\), satisfying

\begin{equation}
\label{eq:mult}
    \alpha\beta=\sum_{i=1}^R \varphi_i(\alpha)\psi_i(\beta)w_i.
\end{equation}

for all \(\alpha,\beta\in\mathbb F_{q^m}\). Equivalently, the tensor \(M_{q^m}\) admits a decomposition
\begin{equation}
\label{eq:tM}
    M_{q^m}
=
\sum_{i=1}^R w_i\otimes \phi_i\otimes \psi_i
\in
\mathbb F_{q^m}\otimes \mathbb F_{q^m}^{\vee}\otimes \mathbb F_{q^m}^{\vee}.
\end{equation}

where $\fqm^\vee$ denotes the dual of $\fqm$ as vector space over $\Fq$. We recall the following. 

\begin{definition}
    An \textbf{algorithm} for $M_{q^m}$ of length $R$ is a collection $\{\phi_i,\psi_i,w_i\}_{i=1}^R$ satisfying~\eqref{eq:mult} or equivalently~\eqref{eq:tM}. We say that such an algorithm is symmetric if \(\varphi_i=\psi_i\) for every \(i\in\{1,\ldots,R\}\).
\end{definition}

Since multiplication in \(\mathbb F_{q^m}\) is commutative, the tensor
\(M_{q^m}\) is symmetric. It is therefore natural to
consider its symmetric tensor rank, denoted by
\(\mu_q^{\mathrm{sym}}(m)\), where one requires decompositions with
\(\varphi_i=\psi_i\) for all \(i\). It is known that $\mu_q(m) \le \mu_q^{\mathrm{sym}}(m)$, and determining exact values or bounds for these quantities is a central topic in algebraic complexity theory (see~\cite{burgisser2013algebraic} and~\cite{ballet2021tensor}). A fundamental lower bound for \(\mu_q(m)\), and hence for
\(\mu_q^{\mathrm{sym}}(m)\), was obtained in
\cite[Theorem~2.2]{ballet2021tensor} and
\cite[Proposition~5.12]{byrne2019tensor} in connection with rank-metric
codes; see also Section~\ref{sec:rankmetriccodes}. It is a consequence of the
results in \cite{de1983characterization} applied to
\cite[Proposition~14.47]{burgisser2013algebraic}, and can be stated as follows.

\begin{theorem}\label{thm:burg}
For every \(q\) and \(m\), one has
\[
\mu_q(m)\geq 2m-1.
\]
Moreover, $\mu_q(m)=\mu_q^{\mathrm{sym}}(m)=2m-1$ if and only if \(q\geq 2m-2\).
\end{theorem}

\begin{remark}\label{rem:burg}
    As already observed in~\cite[Section~9.1]{ballet2021tensor}, the equality case in the theorem above also implies $\mu_q(m)=\mu_q^\sym(m)=2m-1$ if and only if $q\geq 2m-2$.
\end{remark}

Shokrollahi~\cite{shokrollahi1992optimal} established the strict inequality, and Chaumine~\cite{chaumine2006bilinear} proved the following result by applying the D. V. Chudnovsky and G. V. Chudnovsky algorithm~\cite{chudnovsky1988algebraic} to suitably chosen elliptic curves.

\begin{theorem}\label{thm:shok}
    We have $\mu_q^\sym(m)=2m$ if 
    \begin{equation*}
        \frac{1}{2}q+1<m\leq \frac{1}{2}(q+1+\epsilon(q))
    \end{equation*}
    where
    \begin{equation*}
        \epsilon(q)=\begin{cases}
            \text{the greatest integer $\leq 2\sqrt{q}$ prime to $q$} & \text{if $q$ is not a perfect square},\\
            \text{$2\sqrt{q}$ if $q$ is a perfect square}.
        \end{cases}
    \end{equation*}
    In particular, in this case we have $\mu_q(m)=\mu_q^\sym(m)$.
\end{theorem}

Moreover, for values of $m$ not covered by Theorems~\ref{thm:burg} and~\ref{thm:shok}, only a few particular exact values of $\mu_q^\sym(m)$ are known, all of which are obtained in~\cite{chudnovsky1988algebraic}. On the other hand, numerous upper bounds on $\mu_q^\sym(m)$ have been established in the literature. A list of the best known bounds for $q\in\{2,3,4\}$ can be found in~\cite[Table~2]{ballet2021tensor}.We will also use the following matrix-slice description of a tensor. Let
\(t=M_{q^m}\), and fix coordinates so that \(t=(t_{i,j,\ell})\), with
\(i,j,\ell\in\{1,\ldots,m\}\). For each \(i\in\{1,\ldots,m\}\), let
\[
t_i=(t_{i,j,\ell})_{j,\ell}\in \mathbb F_q^{m\times m}.
\]
We call the ordered tuple
\[
\operatorname{ss}_1(t)=(t_1,\ldots,t_m)
\]
the \textbf{first slice} of \(t\). With a slight abuse of notation, we will also write
\(\operatorname{ss}_1(t)\) for the set~\(\{t_1,\ldots,t_m\}\), or for the
\(\mathbb F_q\)-linear space spanned by these matrices, whenever the meaning is clear from
the context. The next result is~\cite[Proposition~14.45]{burgisser2013algebraic} applied to $t$ and gives a characterization of the tensor rank.

\begin{proposition}\label{prop:trk}
    Let $R$ be a positive integer. The following are equivalent.
    \begin{enumerate}
        \item $\mu_q(m)\leq R$.
        \item There exist rank-one matrices $A_1\ldots,A_R\in\fq^{m\times m}$ with $\ss_1(t)\leq\langle A_1\ldots,A_R\rangle$.
    \end{enumerate}
\end{proposition}

In particular, \(\mu_q(m)=R\) if and only if \(R\) is the smallest integer for which
there exist rank-one matrices \(A_1,\ldots,A_R\in\mathbb F_q^{m\times m}\) satisfying the containment above. In this case, we call
\(\{A_1,\ldots,A_R\}\) a \textbf{tensor-rank decomposition} of \(t\).
Similar arguments applied to symmetric tensors yield the characterization below. We recall that $M_{q^m}$ admits a symmetric tensor representation, as it is a symmetric tensor, and therefore its $1$-slice is a sequence of symmetric matrices.

\begin{proposition}\label{prop:symtrk}
    Let $R$ be a positive integer. The following are equivalent.
    \begin{enumerate}
        \item $\mu_q^\sym(m)\leq R$.
        \item There exist symmetric matrices $\hat A_1\ldots,\hat A_R\in\fq^{m\times m}$ of rank one with $\ss_1(t^\sym)\leq\langle \hat A_1\ldots,\hat A_R\rangle$.
    \end{enumerate}
\end{proposition}

As above, when \(R=\mu_q^{\mathrm{sym}}(m)\), we call such a collection
\(\{\widehat A_1,\ldots,\widehat A_R\}\) a \textbf{symmetric tensor-rank decomposition}
of \(M_{q^m}\). Although several ad hoc constructions are known for particular parameters, determining
explicit tensor-rank or symmetric tensor-rank decompositions of \(M_{q^m}\) remains a
difficult open problem in general. In~\cite{kruskal1977three}, Kruskal proved a lower bound on the tensor rank of \(3\)-tensors. In the
language of matrix spaces, this bound can be written as
\[
\mu_q(m)\geq
\dim\big(\langle \operatorname{ss}_1(t)\rangle_{\mathbb F_q}\big)
+
\min\{\operatorname{rk}(M):0\neq M\in
\langle \operatorname{ss}_1(t)\rangle_{\mathbb F_q}\}
-1.
\]
Combining Kruskal's bound with Theorems~2.2 and~2.4, Remark~2.3, \cite[Table~1]{ballet2021tensor} and the
upper bounds collected in \cite[Table~2]{ballet2021tensor}, we obtain the values and intervals reported in
Table~1 for \(\mu_q^{\mathrm{sym}}(m)\), for selected values of \(q\) and \(m\). When the available bounds determine a unique value, we record that value. Otherwise,
we write \(a\)--\(b\) to indicate that \(\mu_q^{\mathrm{sym}}(m)\in\{a,a+1,\ldots,b\}\).
\begin{table}[h]
\centering
\renewcommand{\arraystretch}{1.2}
\begin{tabular}{|c|c|c|c|c|c|c|c|c|c|}
\hline
\diagbox{$q$}{$m$} & 2 & 3 & 4 & 5 & 6 & 7 & 8 & 9 & 10\\\hline
2 & 3 & 6 & 9 & 10 -- 13 & 15 & 14 -- 22 & 16 -- 24 & 18 -- 30 & 20 -- 33\\\hline
3 & 3 & 6 & 8 -- 9 & 10 -- 11 & 12 -- 15 & 14 -- 19 & 16 -- 21 & 18 -- 26 & 20 -- 27 \\\hline
4 & 3 & 5 & 8 & 10 -- 11 & 12 -- 14 & 14 -- 17 & 16 -- 20 & 18 -- 23 & 20 -- 27\\\hline
\end{tabular}
\caption{Known values and intervals for $\mu_q^{\mathrm{sym}}(m)$.}
\label{table:values}
\end{table}

The study of tensor rank and tensor decompositions of \(M_{q^m}\) is central in algebraic complexity theory, but it also has important connections with other areas. One such connection is with coding theory: low tensor-rank codes in the rank metric have been shown to be useful for storage and encoding applications~\cite{byrne2019tensor}. Motivated by Propositions~2.5 and~2.6, one can define the tensor rank, respectively the symmetric tensor rank, of a rank-metric code as the minimum number of rank-one, respectively rank-one symmetric, matrices whose \(\mathbb F_q\)-span contains the code. We return to this point in Section~\ref{sec:rankmetriccodes}. 

In~\cite[Proposition~5.13 and Lemma~5.14]{byrne2019tensor}, the authors show that the first-slice space \(\langle \operatorname{ss}_1(t)\rangle_{\mathbb F_q}\) is equivalent to a \textit{1-dimensional
Gabidulin code}. In the remainder of this section, we recall this connection and reformulate symmetric tensor decompositions of \(M_{q^m}\) in terms of linearized polynomials. The latter form the foundation of the techniques introduced in Section~\ref{sec:symdec}. We denote by $\mathcal{G}_1(m,q)=\langle x\rangle_{\mathbb F_{q^m}}$ the 1-dimensional Gabidulin code. Since every map \(a x\), with
\(a\in\mathbb F_{q^m}\), is self-adjoint with respect to the trace form, we have
\[
\mathcal{G}_1(m,q)\leq \mathcal{S}_q(m).
\]
Equivalently, after choosing a basis of \(\mathbb F_{q^m}\) over \(\mathbb F_q\) and the
corresponding trace-dual basis, the elements of \(\mathcal{G}_1(m,q)\) are represented by symmetric
matrices. On the other hand, with respect to the usual basis representation, \(\mathcal{G}_1(m,q)\) is
isomorphic to
\[
\langle I,M,\ldots,M^{m-1}\rangle_{\mathbb F_q},
\]
where \(M\in\mathbb F_q^{m\times m}\) is the companion matrix of the minimal polynomial
of a primitive element of \(\mathbb F_{q^m}\). We now combine the slice characterization above with the description of rank-one
symmetric linearized polynomials from Proposition~\ref{prop:rank1sym}. Such polynomials correspond,
under the field trace, to rank-one symmetric bilinear forms on
\(\mathbb F_{q^m}\). Equivalently, after choosing a basis and the corresponding
trace-dual basis, they are represented by rank-one symmetric matrices. Using the identification of the first-slice space of \(M_{q^m}\) with
\(\mathcal{G}_1(m,q)=\langle x\rangle_{\mathbb F_{q^m}}\), we obtain the following reformulation of
symmetric tensor decompositions.

\begin{proposition}\label{prop:<=Rwithlinpol}
    Let $R$ be a positive integer. The following are equivalent.
    \begin{enumerate}
        \item $\mu_q^\sym(m)\leq R$.
        \item there exist $\alpha_1,\ldots,\alpha_R,\beta_1,\ldots,\beta_R\in\fqm^*$ such that $\alpha_i/\beta_i\in\Fq$, for any $i\in\{1,\ldots,R\}$, and $\langle x\rangle_{\fqm}\leq\langle \alpha_1\Tr_{q^m/q}(\beta_1 x),\ldots, \alpha_R\Tr_{q^m/q}(\beta_R x)\rangle_{\fq}$.
    \end{enumerate}
\end{proposition}

In the next section, we use Proposition~\ref{prop:<=Rwithlinpol} to develop an explicit procedure for
constructing symmetric tensor decompositions of \(M_{q^m}\) in terms of linearized polynomials. These decompositions can then be translated into matrix, tensor, or bilinear-form decompositions via the correspondences described in Section~\ref{sec:1}.

\section{Symmetric tensor decompositions}\label{sec:symdec}

In this section, we develop a constructive approach to symmetric tensor
decompositions of the multiplication map \(M_{q^m}\) using linearized
polynomials. Starting from Proposition~\ref{prop:<=Rwithlinpol},  we reformulate the search for
a symmetric tensor decomposition as the search for elements of
\(\mathbb F_{q^m}\) satisfying explicit linear conditions over \(\mathbb F_q\). We then exploit the Frobenius automorphism to replace these conditions by equivalent enlarged linear systems. This formulation is well suited for
explicit computations and will be used in Section~\ref{sec:4} to construct concrete decompositions for small extension degrees.

More precisely, Proposition~\ref{prop:<=Rwithlinpol} shows that computing
\(\mu_q^{\mathrm{sym}}(m)\) is equivalent to finding the smallest integer \(R\)
for which there exist $\alpha_1,\ldots,\alpha_R,\beta_1,\ldots,\beta_R\in\mathbb F_{q^m}^{*}$ such that \(\alpha_i/\beta_i\in\mathbb F_q\) for every \(i\), and
\[
\langle x\rangle_{\mathbb F_{q^m}}
\subseteq
\left\langle
\alpha_1\operatorname{Tr}_{q^m/q}(\beta_1x),\ldots,
\alpha_R\operatorname{Tr}_{q^m/q}(\beta_Rx)
\right\rangle_{\mathbb F_q}.
\]
Whenever such elements are found, the corresponding rank-one symmetric
linearized polynomials $\alpha_i\operatorname{Tr}_{q^m/q}(\beta_i x)$, $i\in\{1,\ldots,R\}$, yield a symmetric tensor-rank decomposition of~\(M_{q^m}\). The following proposition translates this condition into coordinates with respect to the standard \(q\)-polynomial basis.

\begin{proposition}\label{prop:linsymtensorrank}
Let \(R\) be a positive integer, and let \(\xi\) be a generator of
\(\mathbb F_{q^m}\) over \(\mathbb F_q\). Then the following are equivalent.
\begin{enumerate}
    \item \(\mu_q^{\mathrm{sym}}(m)\leq R\).
    \item There exist \(\alpha_1,\ldots,\alpha_R\in\mathbb F_{q^m}^{*}\) such that
    \begin{multline*}
        \left\langle
    (1,0,\ldots,0),(\xi,0,\ldots,0),\ldots,(\xi^{m-1},0,\ldots,0)
    \right\rangle_{\mathbb F_q}
   \\ \subseteq
    \left\langle
    (\alpha_1^2,\alpha_1^{1+q},\ldots,\alpha_1^{1+q^{m-1}}),
    \ldots,(\alpha_R^2,\alpha_R^{1+q},\ldots,\alpha_R^{1+q^{m-1}})
    \right\rangle_{\mathbb F_q}.
    \end{multline*}
\end{enumerate}
\end{proposition}
\begin{proof}
By Proposition~\ref{prop:<=Rwithlinpol}, \(\mu_q^{\mathrm{sym}}(m)\le R\) if and only if there exist
\(\alpha_1,\ldots,\alpha_R,\beta_1,\ldots,\beta_R\in\mathbb F_{q^m}^{*}\), with
\(\alpha_i/\beta_i\in\mathbb F_q^{*}\) for every \(i\), such that
\[
\langle x\rangle_{\mathbb F_{q^m}}
\subseteq
\left\langle
\alpha_1\operatorname{Tr}_{q^m/q}(\beta_1x),\ldots,
\alpha_R\operatorname{Tr}_{q^m/q}(\beta_Rx)
\right\rangle_{\mathbb F_q}.
\]
Since \(\alpha_i/\beta_i\in\mathbb F_q^{*}\), we may assume, up to multiplication by
nonzero scalars in \(\mathbb F_q\), that \(\beta_i=\alpha_i\) for every \(i\). Let $\mathcal B=(x,x^q,\ldots,x^{q^{m-1}})$ and consider the coordinate map
\[
c_{\mathcal B}:\mathcal{L}_{m,q}\longrightarrow \mathbb F_{q^m}^m:
a_0x+\cdots+a_{m-1}x^{q^{m-1}}
\longmapsto (a_0,\ldots,a_{m-1}).
\]
This map is injective and \(\mathbb F_q\)-linear. Hence
\[
c_{\mathcal B}(\langle x\rangle_{\mathbb F_{q^m}})
=
\left\langle
(1,0,\ldots,0),(\xi,0,\ldots,0),\ldots,
(\xi^{m-1},0,\ldots,0)
\right\rangle_{\mathbb F_q}.
\]
Moreover, $c_{\mathcal B}\big(\alpha_i\operatorname{Tr}_{q^m/q}(\alpha_i x)\big)
=
(\alpha_i^2,\alpha_i^{1+q},\ldots,\alpha_i^{1+q^{m-1}})$ for every \(i\). Therefore the containment in Proposition~\ref{prop:<=Rwithlinpol} is equivalent to
\begin{multline*}
\left\langle
(1,0,\ldots,0),(\xi,0,\ldots,0),\ldots,(\xi^{m-1},0,\ldots,0)
\right\rangle_{\mathbb F_q}\\
\subseteq
\left\langle
(\alpha_1^2,\alpha_1^{1+q},\ldots,\alpha_1^{1+q^{m-1}}),
\ldots,
(\alpha_R^2,\alpha_R^{1+q},\ldots,\alpha_R^{1+q^{m-1}})
\right\rangle_{\mathbb F_q}.
\end{multline*}
The statement follows.
\end{proof}

Thus, the problem of bounding \(\mu_q^{\mathrm{sym}}(m)\), or of constructing an
explicit symmetric tensor decomposition of \(M_{q^m}\), reduces to finding elements
\(\alpha_1,\ldots,\alpha_R\in\mathbb F_{q^m}^{*}\) satisfying the coordinate
containment in Proposition~3.1. Equivalently, for every \(i\in\{0,\ldots,m-1\}\), one has to solve over
\(\mathbb F_q\) the linear system
\[
\Sigma_i \colon
\begin{pmatrix}
\alpha_1^2 & \alpha_2^2 & \cdots & \alpha_R^{2} \\
\alpha_1^{q+1} & \alpha_2^{q+1} & \cdots & \alpha_R^{q+1} \\
\vdots & \vdots & & \vdots \\
\alpha_1^{q^{m-1}+1} & \alpha_2^{q^{m-1}+1} & \cdots & \alpha_R^{q^{m-1}+1}
\end{pmatrix}
\begin{pmatrix}
x_1\\
x_2\\
\vdots\\
x_R
\end{pmatrix}
=
\begin{pmatrix}
\xi^i\\
0\\
\vdots\\
0
\end{pmatrix}.
\]
Here the unknowns \(x_1,\ldots,x_R\) are required to lie in \(\mathbb F_q\).
The next theorem reformulates this condition by enlarging the systems
\(\Sigma_i\) with their Frobenius conjugates, giving a criterion that is
convenient for explicit computations.

In what follows, we write $b_i=(\xi^i,0,\ldots,0)^{\top}\in \mathbb F_{q^m}^m$. For any integer \(j\), we denote by \(b_i^{[j]}\) the vector obtained by applying the \(q^j\)-Frobenius entrywise, namely $b_i^{[j]}=(\xi^{i q^j},0,\ldots,0)^{\top}$.

\begin{theorem}\label{thm:charactfinal}
Let \(R\) be a positive integer. The following are equivalent.
\begin{enumerate}
    \item \(\mu_q^{\mathrm{sym}}(m)\leq R\).
    \item There exist \(\alpha_1,\ldots,\alpha_R\in\mathbb F_{q^m}^{*}\) such that,
    for every \(i\in\{0,\ldots,m-1\}\), the system \(\Sigma_i\) admits a solution
    in \(\mathbb F_q^R\).
\end{enumerate}

Moreover, let
\[
A=
\begin{pmatrix}
\alpha_1^2 & \alpha_2^2 & \cdots & \alpha_R^2\\
\alpha_1^{q+1} & \alpha_2^{q+1} & \cdots & \alpha_R^{q+1}\\
\vdots & \vdots & & \vdots\\
\alpha_1^{q^{m-1}+1} & \alpha_2^{q^{m-1}+1} & \cdots & \alpha_R^{q^{m-1}+1}
\end{pmatrix}.
\]
If there exist \(\alpha_1,\ldots,\alpha_R\in\mathbb F_{q^m}^{*}\) such that, for every
\(i\in\{0,\ldots,m-1\}\), the enlarged system
\[
\Sigma_i^{*}\colon
\begin{pmatrix}
A\\
A^{[1]}\\
\vdots\\
A^{[m-1]}
\end{pmatrix}
X
=
\begin{pmatrix}
b_i\\
b_i^{[1]}\\
\vdots\\
b_i^{[m-1]}
\end{pmatrix}
\]
admits a solution \(X\in\mathbb F_q^R\), where \(A^{[j]}\) denotes the matrix obtained
from \(A\) by applying the~\(q^j\)-Frobenius entrywise, then $\mu_q^{\mathrm{sym}}(m)\leq R$.
\end{theorem}

\begin{proof}
The equivalence follows directly from Proposition~3.1 and from the definition of the systems \(\Sigma_i\). For the final assertion, note first that, since \(X\in\mathbb F_q^R\), applying the
\(q^j\)-Frobenius to \(AX=b_i\) gives
\[
A^{[j]}X=b_i^{[j]}.
\]
Thus the enlarged system \(\Sigma_i^*\) is obtained by adjoining to \(\Sigma_i\) all its Frobenius conjugates.
Now suppose that, for every \(i\in\{0,\ldots,m-1\}\), the system \(\Sigma_i^*\)
admits a solution \(X\in\mathbb F_q^R\). Then its first block row gives \(AX=b_i\),
so \(\Sigma_i\) admits a solution in \(\mathbb F_q^R\) for every \(i\). By the
equivalence proved above, we get $\mu_q^{\mathrm{sym}}(m)\leq R$.
\end{proof}

In the next section, we apply this criterion to obtain explicit symmetric tensor decompositions for small values of \(m\).

\begin{remark}\label{rem:<=Rfinal}
Clearly, if we can find \(R\) elements \(\alpha_1,\ldots,\alpha_R\) satisfying the condition in Theorem~\ref{thm:charactfinal}, then
\(\mu_q^{\sym}(m)\leq R\). Moreover, the corresponding trace functions determine a
natural space in which to look for a tensor decomposition.
\end{remark}

\section{Decompositions for certain values of $m$}\label{sec:4}

In this section, we apply the framework developed above to construct explicit
symmetric tensor decompositions of the multiplication map \(M_{q^m}\) for small
values of \(m\). By specializing the systems from Theorem~\ref{thm:charactfinal},
we obtain concrete algebraic criteria that can be checked directly. This allows us to recover known values of the symmetric tensor rank in a unified
way and, in some cases, to provide explicit decompositions for parameters not previously covered by the general results. We treat the cases
\(m\in\{2,3,4\}\) separately, as the associated systems become more intricate
as \(m\) increases. The main tools of this section are Theorem~\ref{thm:charactfinal} and Remark~\ref{rem:<=Rfinal}.

\subsection{ Case $m=2$}

By Table \ref{table:values} we know that $\mu_q^{\sym}(2)=\mu_q(2)=3$. We recover this value using the criterion developed in the previous section and obtain
an explicit symmetric tensor-rank decomposition of the tensor \(M_{q^2}\).

\begin{theorem}\label{thm:m=2}
An optimal symmetric tensor-rank decomposition of \(M_{q^2}\) is given by
\begin{equation}\label{eq:3traces1,eta,eta2}
\left\{
\operatorname{Tr}_{q^2/q}(x),\,
\eta\operatorname{Tr}_{q^2/q}(\eta x),\,
\eta^2\operatorname{Tr}_{q^2/q}(\eta^2x)
\right\},
\end{equation}
where \(\eta\in\mathbb F_{q^2}^{*}\) satisfies
\[
\eta^{2q}-\eta^{2q-1}-\eta^{q+1}+\eta^{q-1}+\eta-1\neq 0.
\]
In particular, such an element \(\eta\) always exists.
\end{theorem}
\begin{proof}
By Theorem \ref{thm:burg}, we have $ \mu_q^\sym(2) \geq  \mu_q(2)\geq 3$.
Thus, it is enough to construct a symmetric tensor-rank decomposition of
\(M_{q^2}\) with three elements.

For \(m=2\), we apply Theorem \ref{thm:charactfinal} with \(R=3\). Taking
\(\alpha_1=1\), we seek \(\alpha_2,\alpha_3\in\mathbb F_{q^2}^{*}\) such that, for
\(i\in\{0,1\}\), the enlarged system
\[ \Sigma_i^* \colon \begin{pmatrix}
1 & \alpha_2^2 & \alpha_3^2\\
1 & \alpha_2^{2	q} & \alpha_3^{2q}\\
1 & \alpha_2^{q+1} &  \alpha_3^{q+1}
\end{pmatrix}
\begin{pmatrix} x_1 \\ x_2 \\ x_3 \end{pmatrix} = \begin{pmatrix} \xi^i \\ \xi^{iq} \\ 0 \end{pmatrix},
 \]
admits a solution in
\(\mathbb F_q^3\), where $\xi$ is any generator of $\F_{q^2}$ over $\fq$.
In this case, the coefficient matrix of \(\Sigma_i^*\) is
\[
\begin{pmatrix}
1 & \alpha_2^2 & \alpha_3^2\\
1 & \alpha_2^{2q} & \alpha_3^{2q}\\
1 & \alpha_2^{q+1} & \alpha_3^{q+1}
\end{pmatrix}.
\]
Hence it is enough to choose \(\alpha_2,\alpha_3\) so that this matrix is nonsingular, that is, if its determinant in nonzero. We now choose \(\alpha_2=\eta\) and \(\alpha_3=\eta^2\). The determinant of the
coefficient matrix is then
\[
\det
\begin{pmatrix}
1 & \eta^2 & \eta^4\\
1 & \eta^{2q} & \eta^{4q}\\
1 & \eta^{q+1} & \eta^{2(q+1)}
\end{pmatrix}
=
\eta^{q+2}
\left(
\eta^{2q}-\eta^{2q-1}-\eta^{q+1}+\eta^{q-1}+\eta-1
\right).
\]
Thus, if
\[
\eta^{2q}-\eta^{2q-1}-\eta^{q+1}+\eta^{q-1}+\eta-1\neq 0,
\]
then the coefficient matrix is nonsingular. It remains to show that such an element \(\eta\) exists. If \(q>2\), consider
\[
F(T)=T^{q+2}
\left(
T^{2q}-T^{2q-1}-T^{q+1}+T^{q-1}+T-1
\right)\in \mathbb F_q[T].
\]
A nonzero root in $\F_{q^2}$ of $F(T)$ needs to be a root of $T^{2q}-T^{2q-1}-T^{q+1}+T^{q-1}+T-1$, therefore for $q>2$  the polynomial \(F(T)/T^{q+2}\) has degree \(2q<q^2-q\). Since \(\mathbb F_{q^2}\setminus\mathbb F_q\) has~\(q^2-q\) elements, \(F\) cannot
vanish on all of \(\mathbb F_{q^2}\setminus\mathbb F_q\). If \(q=2\), the determinant reduces to
\[
\eta^4(\eta-1),
\]
using \(\eta^4=\eta\) on \(\mathbb F_4\). Thus any
\(\eta\in\mathbb F_4\setminus\mathbb F_2\) works. For each \(i\in\{0,1\}\), the unique solution of \(\Sigma_i^*\) is given by
Cramer's rule. Namely,
\[\left( \frac{\det\begin{pmatrix}
\xi^i & \alpha_2^2 & \alpha_3^2\\
\xi^{iq} & \alpha_2^{2	q} & \alpha_3^{2q}\\
0 & \alpha_2^{q+1} &  \alpha_3^{q+1}
\end{pmatrix}}{\det\begin{pmatrix}
1 & \alpha_2^2 & \alpha_3^2\\
1 & \alpha_2^{2	q} & \alpha_3^{2q}\\
1 & \alpha_2^{q+1} &  \alpha_3^{q+1}
\end{pmatrix}}, \frac{\det\begin{pmatrix}
\alpha_1^2 & \xi^i & \alpha_3^2\\
\alpha_1^{2	q} & \xi^{iq} & \alpha_3^{2q}\\
\alpha_1^{q+1} & 0 &   \alpha_3^{q+1}
\end{pmatrix}}{\det\begin{pmatrix}
1 & \alpha_2^2 & \alpha_3^2\\
1 & \alpha_2^{2	q} & \alpha_3^{2q}\\
1 & \alpha_2^{q+1} &  \alpha_3^{q+1}
\end{pmatrix}},
\frac{\det\begin{pmatrix}
\alpha_1^2 & \alpha_2^2 & \xi^i \\
\alpha_1^{2	q}  & \alpha_2^{2q}& \xi^{iq}\\
\alpha_1^{q+1}  &   \alpha_2^{q+1} & 0
\end{pmatrix}}{\det\begin{pmatrix}
1 & \alpha_2^2 & \alpha_3^2\\
1 & \alpha_2^{2	q} & \alpha_3^{2q}\\
1 & \alpha_2^{q+1} &  \alpha_3^{q+1}
\end{pmatrix}}
\right).\]
Moreover, \(X_i\in\mathbb F_q^3\). Indeed, the system \(\Sigma_i^*\) is stable under
the \(q\)-Frobenius map, and the solution is unique because \(D\neq 0\). Hence the
Frobenius conjugate of \(X_i\) is again a solution, and therefore it must coincide
with \(X_i\). Thus \(X_i\) is fixed by Frobenius and lies in \(\mathbb F_q^3\). By Theorem~\ref{thm:charactfinal}, the set
\[
\left\{
\operatorname{Tr}_{q^2/q}(x),\,
\eta\operatorname{Tr}_{q^2/q}(\eta x),\,
\eta^2\operatorname{Tr}_{q^2/q}(\eta^2x)
\right\}
\]
gives a symmetric tensor-rank decomposition of \(M_{q^2}\). This and the lower
bound~\(\mu_q^{\sym}(2)\geq 3\), prove the claim.
\end{proof}

\begin{remark}
When \(q=2\), the condition in Theorem~\ref{thm:m=2} is satisfied by any
\(\eta\in\mathbb F_4\setminus\mathbb F_2\).
\end{remark}

We now illustrate how the trace-function decomposition in Theorem~\ref{thm:m=2}
can be converted into a matrix decomposition.

\begin{example}
Consider \(q=3\), and let \(\alpha\in\mathbb F_9\) be a root of
\(x^2+2x+2\). The ordered basis~\(\mathcal{B}=(1,\alpha)\) of \(\mathbb F_9\) over \(\mathbb F_3\) has
trace-dual basis \(\mathcal{B}^*=(\alpha,\alpha^2)\).
In this case, the first-slice space \(\operatorname{ss}_1(M_9)\) is spanned by
\[ 
X_1=\begin{pmatrix}
    0&1\\1&1
\end{pmatrix}\qquad\textup{ and }\qquad
X_2=\begin{pmatrix}
    1&1\\1&2
\end{pmatrix}.\]
Take \(\eta=\alpha^5\). Then
\[
\eta^{2q}-\eta^{2q-1}-\eta^{q+1}+\eta^{q-1}+\eta-1=\alpha\neq 0,
\]
so Theorem~\ref{thm:m=2} applies. The three rank-one symmetric matrices associated with the trace functions $\operatorname{Tr}_{9/3}(x)$, $
\eta\operatorname{Tr}_{9/3}(\eta x)$,
$\eta^2\operatorname{Tr}_{9/3}(\eta^2 x)$ are obtained by evaluating these linearized polynomials on the trace-dual basis
\(\mathcal{B}^*\) and expanding the results with respect to \(\mathcal{B}\). They~are
\[ A_1=\begin{pmatrix}
    1&0\\0&0
\end{pmatrix},\qquad A_2=\begin{pmatrix}
    0&0\\0&1
\end{pmatrix},\qquad A_3=\begin{pmatrix}
    1&1\\1&1
\end{pmatrix}. \]
Indeed, $X_1=A_3-A_1$ and $X_2=A_3+A_2$. Hence $\operatorname{ss}_1(M_9)= \langle X_1,X_2\rangle_{\F_3}
\subseteq
\langle A_1,A_2,A_3\rangle_{\mathbb F_3}$.
\end{example}

\subsection{Case $m=3$}

When \(m=3\) and \(q\ge 4\), Theorem~\ref{thm:burg} and Remark~\ref{rem:burg} imply that $\mu_q^{\mathrm{sym}}(3)=\mu_q(3)=5$. By Table~\ref{table:values}, for \(q=2\) and \(q=3\) we have
\[
\mu_2^{\mathrm{sym}}(3)=\mu_2(3)=6
\qquad\text{ and }\qquad
\mu_3^{\mathrm{sym}}(3)=\mu_3(3)=6.
\]
In what follows, we provide a method for finding a decomposition of the symmetric tensor; this decomposition is optimal when \(q\in\{2,3\}\).

\begin{theorem}\label{thm:m=3}
    If there exist $\alpha_1,\alpha_2,\ldots,\alpha_6 \in \F_{q^3}^*$ such that
    \[\det
    \begin{pmatrix}
        \alpha_1^2 & \alpha_2^2 & \alpha_3^2 & \alpha_4^2 & \alpha_5^2 & \alpha_6^2\\
        \alpha_1^{2q} & \alpha_2^{2q} & \alpha_3^{2q} & \alpha_4^{2q} & \alpha_5^{2q} & \alpha_6^{2q}\\
        \alpha_1^{2q^2} & \alpha_2^{2q^2} & \alpha_3^{2q^2} & \alpha_4^{2q^2} & \alpha_5^{2q^2} & \alpha_6^{2q^2}\\
        \alpha_1^{q+1} & \alpha_2^{q+1} & \alpha_3^{q+1} & \alpha_4^{q+1} & \alpha_5^{q+1} & \alpha_6^{q+1}\\
        \alpha_1^{q^2+q} & \alpha_2^{q^2+q} & \alpha_3^{q^2+q} & \alpha_4^{q^2+q} & \alpha_5^{q^2+q} & \alpha_6^{q^2+q}\\
        \alpha_1^{q^2+1} & \alpha_2^{q^2+1} & \alpha_3^{q^2+1} & \alpha_4^{q^2+1} & \alpha_5^{q^2+1} & \alpha_6^{q^2+1}
    \end{pmatrix} \ne 0,
    \]
    then \(\mu_q^{\mathrm{sym}}(3)\le 6\). Moreover, the elements \(\alpha_1,\alpha_2,\ldots,\alpha_6\) define a symmetric tensor-rank decomposition of
\(M_{q^3}\).
\end{theorem}
\begin{proof}
The matrix appearing in the statement is the coefficient matrix of the
enlarged system \(\Sigma_i^*\) described in Theorem~\ref{thm:charactfinal}
for \(m=3\). If its determinant is nonzero, then, for every~\(i\in\{0,1,2\}\), the
corresponding system \(\Sigma_i^*\) has a unique solution. Since
\(\Sigma_i^*\) is stable under the \(q\)-Frobenius map, this solution is
fixed by Frobenius and therefore lies in \(\mathbb F_q^6\). The statement
then follows from Theorem~\ref{thm:charactfinal}.
\end{proof}

As a special choice of the elements in Theorem~\ref{thm:m=3}, one may
consider consecutive powers of a single element, namely $\alpha_i=\xi^{i-1}$,  $i=\{1,\ldots,6\}$, for some \(\xi\in \mathbb F_{q^3}\). In the following, we let $T$ be an indeterminate.

\begin{corollary}\label{cor:m=3}
A symmetric tensor-rank
decomposition of \(M_{q^3}\) is
\[ \{\xi^i\mathrm{Tr}_{q^3/q}(\xi^i x) : i \in \{0,\ldots,5\}\}, \]
where $\xi$ is not a root of 
\begin{equation}\label{eq:f(T)}
f(T)=
\det
\begin{pmatrix}
1 & T^2 & (T^2)^2 & (T^3)^2 & (T^4)^2 & (T^5)^2\\
1 & T^{2q} & (T^2)^{2q} & (T^3)^{2q} & (T^4)^{2q} & (T^5)^{2q}\\
1 & T^{2q^2} & (T^2)^{2q^2} & (T^3)^{2q^2} & (T^4)^{2q^2} & (T^5)^{2q^2}\\
1 & T^{q+1} & (T^2)^{q+1} & (T^3)^{q+1} & (T^4)^{q+1} & (T^5)^{q+1}\\
1 & T^{q^2+q} & (T^2)^{q^2+q} & (T^3)^{q^2+q} & (T^4)^{q^2+q} & (T^5)^{q^2+q}\\
1 & T^{q^2+1} & (T^2)^{q^2+1} & (T^3)^{q^2+1} & (T^4)^{q^2+1} & (T^5)^{q^2+1}
\end{pmatrix}
\in \mathbb F_q[T].
\end{equation}
This decomposition is optimal for $q \in \{2,3\}$.
\end{corollary}
\begin{proof}
Observe that every element of \(\mathbb F_q\) is a root of \(f(T)\). Indeed,
if \(a\in\mathbb F_q\), then \(a^q=a\), and the first three rows of the
matrix defining \(f(a)\) coincide. Thus our aim is to find a non-root of~\(f(T)\) in \(\mathbb F_{q^3}\setminus\mathbb F_q\) and therefore it is sufficient to check that $\deg f(T)<q^3$ .

A direct computation shows that, if \(q>17\), then
\[
\deg f(T)=17q^2+9q+4
\]
and the leading monomial is \(-T^{17q^2+9q+4}\). Since
\[
17q^2+9q+4<q^3
\]
for \(q>17\), the polynomial \(f(T)\) cannot vanish on all elements of
\(\mathbb F_{q^3}\). Since all elements of~\(\mathbb F_q\) are roots of
\(f(T)\), there exists $\xi\in \mathbb F_{q^3}\setminus\mathbb F_q$ such that \(f(\xi)\ne 0\).

It remains to consider the prime powers \(q\le 17\). In these cases, we reduce
\(f(T)\) modulo~\(T^{q^3}-T\), since two polynomials define the same function
on \(\mathbb F_{q^3}\) if and only if they are congruent modulo \(T^{q^3}-T\).
The leading terms of these reductions are listed in Table~\ref{table:m=3}.
In each case, the reduction of \(f(T)\) modulo \(T^{q^3}-T\) is nonzero.
Therefore it has degree strictly smaller than \(q^3\) and cannot vanish on all
of \(\mathbb F_{q^3}\). Since every element of \(\mathbb F_q\) is a root of
\(f(T)\), there exists
\[
\xi\in \mathbb F_{q^3}\setminus\mathbb F_q
\]
such that \(f(\xi)\ne 0\). Taking $\alpha_i=\xi^{i-1}$, $i=\{1,\ldots,6\}$, the determinant in Theorem~\ref{thm:m=3} is equal to \(f(\xi)\), and hence is
nonzero. Therefore Theorem~\ref{thm:m=3} gives a symmetric tensor-rank
decomposition of \(M_{q^3}\) with six elements. The equalities for
\(q=2\) and \(q=3\) follow from the lower bounds recorded in
Table~\ref{table:values}.
\end{proof}

\begin{table}[h]
\centering
\renewcommand{\arraystretch}{1.2}
\resizebox{0.89\textwidth}{!}{
\begin{tabular}{|c|c|c|c|c|c|c|c|c|c|c|c|}
\hline
$q$ 
& 2 & 3 & 4 & 5 & 7 & 8 & 9 & 11 & 13 & 16 & 17 \\ \hline
$q^3$ 
& 8 & 27 & 64 & 125 & 343 & 512 & 729 & 1331 & 2197 & 4096 & 4913 \\ \hline
LM
& $T^{6}$ 
& $2 T^{24}$ 
& $ T^{63}$ 
& $ T^{122}$ 
& $2 T^{336}$ 
& $ T^{485}$ 
& $ T^{718}$ 
& $10 T^{1280}$ 
& $12 T^{2178}$ 
& $ T^{4035}$ 
& $ T^{4814}$\\\hline
\end{tabular}}
\caption{\label{table:m=3}Leading term of $f(T)\bmod (T^{q^3}-T)$ for $m=3$ and $q\leq 17$.}
\end{table}

\subsection{Case $m=4$}

In this subsection, we specialize Theorem~\ref{thm:charactfinal} to the case \(m=4\).
Let \(\xi\in \mathbb F_{q^4}\) be an element of degree \(4\) over \(\mathbb F_q\), so that
\(\{1,\xi,\xi^2,\xi^3\}\) form an \(\mathbb F_q\)-basis of \(\mathbb F_{q^4}\).

\begin{theorem}\label{thm:m=4}
Let \(R\) be a positive integer, and let $\alpha_1,\ldots,\alpha_R\in \mathbb F_{q^4}^*$. Consider the matrix $A\in\F_{q^4}^{10\times R}$ and the vector $b_i\in \F_{q^4}^{10}$, for every \(i\in\{0,1,2,3\}\), defined as 
\[
A=
\begin{pmatrix}
\alpha_1^2 & \alpha_2^2 & \cdots & \alpha_R^2 \\
\alpha_1^{2q} & \alpha_2^{2q} & \cdots & \alpha_R^{2q} \\
\alpha_1^{2q^2} & \alpha_2^{2q^2} & \cdots & \alpha_R^{2q^2} \\
\alpha_1^{2q^3} & \alpha_2^{2q^3} & \cdots & \alpha_R^{2q^3} \\
\alpha_1^{q+1} & \alpha_2^{q+1} & \cdots & \alpha_R^{q+1} \\
\alpha_1^{q^2+q} & \alpha_2^{q^2+q} & \cdots & \alpha_R^{q^2+q} \\
\alpha_1^{q^3+q^2} & \alpha_2^{q^3+q^2} & \cdots & \alpha_R^{q^3+q^2} \\
\alpha_1^{q^3+1} & \alpha_2^{q^3+1} & \cdots & \alpha_R^{q^3+1} \\
\alpha_1^{q^2+1} & \alpha_2^{q^2+1} & \cdots & \alpha_R^{q^2+1} \\
\alpha_1^{q^3+q} & \alpha_2^{q^3+q} & \cdots & \alpha_R^{q^3+q}
\end{pmatrix}\qquad\text{ and }\qquad
b_i=
\begin{pmatrix}
\xi^i\\
\xi^{iq}\\
\xi^{iq^2}\\
\xi^{iq^3}\\
0\\
0\\
0\\
0\\
0\\
0
\end{pmatrix}.
\]
If, for every \(i\in\{0,1,2,3\}\), the linear system
\[
AX=b_i
\]
admits a solution \(X\in\mathbb F_q^R\), then
\[
\mu_q^{\mathrm{sym}}(4)\le R.
\]
Moreover, the elements \(\alpha_1,\ldots,\alpha_R\) yield the symmetric tensor decomposition
\[
\left\{
\alpha_1\operatorname{Tr}_{q^4/q}(\alpha_1x),
\ldots,
\alpha_R\operatorname{Tr}_{q^4/q}(\alpha_Rx)
\right\}.
\]
\end{theorem}
\begin{proof}
For \(m=4\), Theorem~\ref{thm:charactfinal} starts from the requirement that,
for every \(i\in\{0,1,2,3\}\), the system
\[
\begin{pmatrix}
\alpha_1^2 & \alpha_2^2 & \cdots & \alpha_R^2 \\
\alpha_1^{q+1} & \alpha_2^{q+1} & \cdots & \alpha_R^{q+1} \\
\alpha_1^{q^2+1} & \alpha_2^{q^2+1} & \cdots & \alpha_R^{q^2+1} \\
\alpha_1^{q^3+1} & \alpha_2^{q^3+1} & \cdots & \alpha_R^{q^3+1}
\end{pmatrix}
X
=
\begin{pmatrix}
\xi^i\\
0\\
0\\
0
\end{pmatrix}
\]
admits a solution \(X\in\mathbb F_q^R\). Since the entries of \(X\) are
required to lie in \(\mathbb F_q\), applying the \(q\)-Frobenius automorphism
to the equations gives the Frobenius-conjugate equations. Thus the enlarged
system in Theorem~\ref{thm:charactfinal} is obtained by adjoining the
Frobenius conjugates of the rows above.

The Frobenius orbit of the exponent \(2\) is
\[
2,\ 2q,\ 2q^2,\ 2q^3.
\]
The Frobenius orbit of the exponent \(q+1\) is
\[
q+1,\ q^2+q,\ q^3+q^2,\ q^3+1.
\]
Finally, the Frobenius orbit of the exponent \(q^2+1\) is
\[
q^2+1,\ q^3+q.
\]
The exponent \(q^3+1\) already belongs to the orbit of \(q+1\). Therefore,
the distinct Frobenius-conjugate rows are precisely the rows of the matrix
\(A\) appearing in the statement, and the corresponding right-hand side is
the vector \(b_i\).

Hence, if for every \(i\in\{0,1,2,3\}\) the system $AX=b_i$ admits a solution \(X\in\mathbb F_q^R\), then the hypothesis of
Theorem~\ref{thm:charactfinal} is satisfied. Consequently,
\[
\mu_q^{\mathrm{sym}}(4)\le R.
\]
The associated rank-one symmetric linearized polynomials are $\alpha_j\operatorname{Tr}_{q^4/q}(\alpha_jx)$, $j\in\{1,\ldots,R\}$, and therefore the elements \(\alpha_1,\ldots,\alpha_R\) yield the symmetric
tensor decomposition
\[
\left\{
\alpha_1\operatorname{Tr}_{q^4/q}(\alpha_1x),
\ldots,
\alpha_R\operatorname{Tr}_{q^4/q}(\alpha_Rx)
\right\}.\qedhere
\]
\end{proof}  

\begin{remark}
For a fixed \(i\), the rank condition $\operatorname{rk}_{\mathbb F_{q^4}}(A)
=
\operatorname{rk}_{\mathbb F_{q^4}}(A \mid b_i)$ only guarantees that the system \(AX=b_i\) has a solution over
\(\mathbb F_{q^4}\). In Theorem~\ref{thm:m=4}, however, we need a
solution \(X\in\mathbb F_q^R\). Thus solvability over the extension field is
not sufficient, and the theorem is stated directly in terms of
\(\mathbb F_q\)-solutions.
\end{remark}

The following result provides explicit symmetric tensor decompositions for \(q\in\{2,3,4,5\}\) and \(m=4\). For \(q=2,4,5\), the resulting decompositions
attain the values recorded in Table~\ref{table:values}.

\begin{corollary}\label{cor:m=4}
The following hold.
\begin{enumerate}
    \item An optimal symmetric
    tensor decomposition of \(M_{2^4}\) is
    \[
    \left\{
    \xi^{i_1}\operatorname{Tr}_{2^4/2}(\xi^{i_1}x),
    \ldots,
    \xi^{i_9}\operatorname{Tr}_{2^4/2}(\xi^{i_9}x)
    \right\},
    \]
    where \(\xi\) and \(\{i_1,\ldots,i_9\}\) are as in
    Table~\ref{table:m4}.

    \item An optimal symmetric tensor decomposition of \(M_{3^4}\) is
    \[
    \left\{
    \xi^{i_1}\operatorname{Tr}_{3^4/3}(\xi^{i_1}x),
    \ldots,
    \xi^{i_9}\operatorname{Tr}_{3^4/3}(\xi^{i_9}x)
    \right\},
    \]
    where \(\xi\) and \(\{i_1,\ldots,i_9\}\) are as in
    Table~\ref{table:m4}.

    \item  For $q\in\{4,5\}$, an optimal symmetric tensor decomposition of~\(M_{q^4}\) is
    \[
    \left\{
    \xi^{i_1}\operatorname{Tr}_{q^4/q}(\xi^{i_1}x),
    \ldots,
    \xi^{i_8}\operatorname{Tr}_{q^4/q}(\xi^{i_8}x)
    \right\},
    \]
    where \(\xi\) and \(\{i_1,\ldots,i_8\}\) are as in
    Table~\ref{table:m4}.
\end{enumerate}
\end{corollary}

\begin{proof}
For each row of Table~\ref{table:m4}, let
\[
\alpha_j=\xi^{i_j},
\]
where the exponents \(i_j\) are those listed in the corresponding row. A MAGMA
computation verifies that, for these choices of \(\alpha_j\), the systems
\[
AX=b_i,\qquad i\in\{0,1,2,3\},
\]
admit solutions in \(\mathbb F_q^R\), where \(R=9\) for \(q\in\{2,3\}\) and
\(R=8\) for \(q\in\{4,5\}\). Therefore, by Theorem~\ref{thm:m=4}, the
listed elements yield symmetric tensor decompositions of \(M_{q^4}\) with the
claimed number of elements.

For \(q=2\), Table~\ref{table:values} gives $\mu^{\mathrm{sym}}_2(4)=9$, so the decomposition in the statement is optimal. Similarly, for \(q=4\),
Table~\ref{table:values} gives $\mu^{\mathrm{sym}}_4(4)=8$, so the corresponding decomposition is optimal. Finally, for \(q=5\),
Theorem~\ref{thm:shok} gives $\mu^{\mathrm{sym}}_5(4)=2\cdot 4=8$, and hence the decomposition in the statement is optimal in this case as well. This concludes the proof.
\end{proof}

\begin{table}[h]
    \centering
    \begin{tabular}{|c|c|c|}\hline
        $q$ &  Minimal Polynomial of $\xi$ & Sequence\\\hline
        2 & $x^4 + x + 1$ & $[ 0, 1, 4, 5, 6, 9, 10, 11, 14 ]$\\\hline
        3 & $x^4+2x^3+2$ & $[ 0, 9, 15, 33, 36, 42, 52, 54, 70 ]$\\\hline
        \multirow{2}{*}{4} & $x^4+x^3+\beta x^2+ \beta x+\beta$ & \multirow{2}{*}{$[ 0, 9, 15, 33, 36, 42, 52, 54, 70 ]$}\\
         & where $\beta$ is a root of $x^2+x+1$ & \\\hline
        5 & $x^4+4x^2+4x+2$ & $[ 9, 63, 104, 170, 419, 487, 500, 542 ]$\\\hline
    \end{tabular}
    \captionsetup{width=0.85\textwidth}
    \caption{Minimal polynomial, over the base field, of \(\xi\) and exponents \(i_j\) used for the \(m=4\) tensor decompositions. }
    \label{table:m4}
\end{table}

\begin{remark}
For \(q\in\{2,4,5\}\), the decompositions in Corollary~\ref{cor:m=4} are optimal.
For \(q=3\), the corollary provides an explicit symmetric tensor decomposition of
\(M_{3^4}\) with nine elements. Several computational searches in MAGMA did not yield a
decomposition with eight elements, although a complete exhaustive search was not feasible with our implementation.
\end{remark}

\section{Symmetric tensor rank of symmetric rank-metric codes}\label{sec:rankmetriccodes}

Let \(\Sym_q(m)\) denote the space of symmetric matrices in
\(\mathbb F_q^{m\times m}\). The rank distance between two matrices
\(A,B\in\Sym_q(m)\) is defined by
\[
d(A,B)=\rk(A-B).
\]
For an \(\mathbb F_q\)-linear symmetric rank-metric code
\(C\subseteq \Sym_q(m)\), its minimum distance is
\[
d(C)=\min\{d(A,B):A,B\in C,\ A\ne B\}.
\]
Equivalently, since \(C\) is linear,
\[
d(C)=\min\{\rk(A):A\in C\setminus\{0\}\}.
\]
An \(\mathbb F_q\)-linear symmetric rank-metric code
\(C\leq \Sym_q(m)\) is described by the parameters $[m,\dim_{\mathbb F_q}(C),d(C)]_q$, where \(d(C)\) denotes its minimum distance. These parameters are constrained by the following Singleton-type bound.

\begin{theorem}[\text{\cite[Theorem 3.3]{schmidt2015symmetric}}]\label{thm:SingletonBound}
Let \(C\) be an \(\mathbb F_q\)-linear symmetric rank-metric code in
\(\Sym_q(m)\) with minimum distance \(d\). Then
\[
\dim_{\mathbb F_q}(C) \le
\begin{cases}
\dfrac{m(m-d+2)}{2}, & \text{if \(m-d\) is even},\\[1ex]
\dfrac{(m+1)(m-d+1)}{2}, & \text{if \(m-d\) is odd}.
\end{cases}
\]
\end{theorem}

As shown in \cite{schmidt2015symmetric}, this bound is tight. This motivates the following definition.

\begin{definition}
Let \(C\leq \Sym_q(m)\) be an \(\mathbb F_q\)-linear symmetric rank-metric code. We say that \(C\) is a \textbf{symmetric Maximum Rank Distance code},
or simply a \textbf{symmetric MRD code}, if equality holds in
Theorem~\ref{thm:SingletonBound}.
\end{definition}

Symmetric MRD codes exist for all admissible choices of \(q\), \(m\), and the minimum distance~\(d\); see~\cite[Theorem 4.4]{schmidt2015symmetric}. An equivalent description of these constructions in terms of linearized polynomials can be found in~\cite{longobardi2020automorphism}.

\begin{theorem}[\text{\cite[Theorem 4.4]{schmidt2015symmetric}}]\label{Th-Schmidt-4.4}
Let \(m\) and \(d\) be positive integers such that \(1\le d\le m\) and
\(m-d\) is even. Then
\[
S_{q,m,d}=
\left\{
b_0x+
\sum_{j=1}^{\frac{m-d}{2}}
\left(
b_jx^{q^j}+(b_jx)^{q^{m-j}}
\right)
:
b_0,\ldots,b_{\frac{m-d}{2}}\in \mathbb F_{q^m}
\right\}
\]
defines a symmetric MRD code in \(\Sym_q(m)\).
\end{theorem}

Furthermore, \cite[Theorem~4.1]{schmidt2015symmetric} shows that symmetric
MRD codes in \(\Sym_q(m)\) with minimum distance \(d\) and \(m-d\) odd can be
obtained by \textit{puncturing} suitable symmetric MRD codes of the form described above, that is, by deleting appropriate rows and columns in a matrix representation.

When \(d=3\) and \(m\) is odd, these codes are perfect; see
\cite{mushrraf2025perfect}. In addition, efficient encoding and decoding
algorithms for these families of codes have been developed; see
\cite{gabidulin2002representation,gabidulin2004symmetric,
gabidulin2006symmetric,kadir2022encoding}. Other constructions of symmetric MRD codes are presented in
\cite[Section~5]{longobardi2020automorphism},
\cite{zhou2020equivalence,bik2024higher,tang2026new}; see also
\cite{byrne2026q,alnajjarine2026linear} for connections with conics and
matroids. Following \cite{byrne2019tensor,byrne2023tensor}, we define the symmetric
tensor rank of a symmetric rank-metric code.

\begin{definition}\label{def:symtrk}
Let \(C\leq \Sym_q(m)\) be an \(\mathbb F_q\)-linear symmetric rank-metric code. The minimum number \(R\) of rank-one matrices
\(A_1,\ldots,A_R\in \Sym_q(m)\) such that
\[
C\leq \langle A_1,\ldots,A_R\rangle_{\mathbb F_q}
\]
is called the \textbf{symmetric tensor rank} of \(C\). We denote it by
\(\operatorname{strk}(C)\).
\end{definition}

\begin{remark}
This definition is compatible with the usual
notion of symmetric tensor rank for symmetric tensors. Indeed, if \(X\) is a
symmetric tensor and \(C=\langle \ss_1(X)\rangle_{\mathbb F_q}\) is its first-slice
space, then Proposition~\ref{prop:symtrk} shows that
the symmetric tensor rank of \(X\) is the smallest integer \(R\) for which there
exist rank-one symmetric matrices \(A_1,\ldots,A_R\in\Sym_q(m)\) satisfying
\[
C\leq \langle A_1,\ldots,A_R\rangle_{\mathbb F_q}.
\]
Thus, when a symmetric rank-metric code arises as the first-slice space of a
symmetric tensor, its symmetric tensor rank agrees with the usual symmetric
tensor rank of that tensor. In this sense, the definition extends symmetric
tensor rank from individual symmetric tensors to arbitrary symmetric
rank-metric codes.
\end{remark}

\begin{remark}
The symmetric tensor rank is invariant under the natural equivalence of
symmetric rank-metric codes. Recall that the rank-preserving isometries of
\(\Sym_q(m)\) are given by
\[
\varphi_P:\Sym_q(m)\longrightarrow \Sym_q(m):
A\longmapsto P^\top A P,
\]
where \(P\in \GL\); see~\cite{wan1996geometry}. Thus two symmetric
rank-metric codes \(C,C'\subseteq \Sym_q(m)\) are \textbf{equivalent} if
\(C'=\varphi_P(C)\) for some \(P\in \GL\). Suppose that
\[
C\leq \langle A_1,\ldots,A_R\rangle_{\mathbb F_q},
\]
where \(A_1,\ldots,A_R\in\Sym_q(m)\) have rank one. Applying \(\varphi_P\), we get
\[
C'\leq
\langle P^\top A_1P,\ldots,P^\top A_RP\rangle_{\mathbb F_q}.
\]
Since congruence preserves symmetry and rank, each \(P^\top A_jP\) is again a
rank-one symmetric matrix. Hence any symmetric rank-one spanning set for \(C\)
gives one of the same size for \(C'\). Applying the same argument to
\(\varphi_{P^{-1}}\), we obtain
\[
\operatorname{strk}(C)=\operatorname{strk}(C').
\]
Thus \(\operatorname{strk}\) is an invariant of equivalence classes of
symmetric rank-metric codes.
\end{remark}

The symmetric tensor rank is the symmetric analogue of the usual tensor rank of
a rank-metric code. Recall that, for a rank-metric code
\(\mathcal C\subseteq \mathbb F_q^{m\times m}\), its tensor rank is the minimum
integer \(R\) such that there exist rank-one matrices
\(A_1,\ldots,A_R\in \mathbb F_q^{m\times m}\) satisfying
\[
\mathcal C\subseteq \langle A_1,\ldots,A_R\rangle_{\mathbb F_q}.
\]
For a symmetric rank-metric code \(\mathcal C\subseteq \Sym_q(m)\), the invariant
\(\operatorname{strk}(\mathcal C)\) is obtained by imposing the additional
requirement that the rank-one matrices \(A_1,\ldots,A_R\) be symmetric. Hence, we have
\[
\operatorname{trk}(\mathcal C)\le \operatorname{strk}(\mathcal C).
\]

The codes \(\mathcal S_{q,m,d}\) can be viewed as symmetric
analogues of Gabidulin codes. In the extremal case \(d=m\), they coincide with
the one-dimensional Gabidulin code; see
\cite[Proposition~3.1]{longobardi2020automorphism}. For \(d<m\), however, these
codes are not Gabidulin codes in general.

The results obtained in the previous sections can be reformulated naturally in
the language of symmetric rank-metric codes. Indeed, for \(d=m\), the code
\[
\mathcal S_{q,m,m}=\langle x\rangle_{\mathbb F_{q^m}}\leq S_q(m)
\]
is the one-dimensional Gabidulin code, viewed through the correspondence of
Section~\ref{sec:1} as an \(\mathbb F_q\)-linear symmetric
rank-metric code. Moreover, the first-slice space of the multiplication tensor
\(M_{q^m}\) can be identified with \(\mathcal S_{q,m,m}\). Therefore,
Proposition~\ref{prop:symtrk} gives
\[
\operatorname{strk}(\mathcal S_{q,m,m})=\mu_q^{\mathrm{sym}}(m).
\]
Equivalently, \(\mu_q^{\mathrm{sym}}(m)\) is the minimum integer \(R\) such that
\[
\mathcal S_{q,m,m}
\leq
\langle f_1,\ldots,f_R\rangle_{\mathbb F_q},
\]
where \(f_1,\ldots,f_R\in S_q(m)\) are rank-one symmetric linearized
polynomials.

Thus, the explicit decompositions constructed above can be read equivalently as
decompositions of \(M_{q^m}\) or as symmetric rank-one spanning sets for
\(\mathcal S_{q,m,m}\). The next two propositions make this correspondence
precise.

\begin{proposition}
With the notation above, the following hold.
\begin{enumerate}
    \item \(\operatorname{strk}(\mathcal S_{q,m,m})=\mu_q^{\mathrm{sym}}(m)\).
    \item \(\operatorname{strk}(\mathcal S_{q,m,m})\ge 2m-1\), with equality if and only if \(q\ge 2m-2\).
    \item If
    \[
    \frac q2+1<m\le \frac{1}{2}(q+1+\epsilon(q)),
    \]
    then
    \[
    \operatorname{strk}(\mathcal S_{q,m,m})=2m.
    \]
\end{enumerate}
\end{proposition}

As a consequence of Theorems~\ref{thm:m=2}, \ref{thm:m=3}, and~\ref{thm:m=4}, we obtain explicit symmetric rank-one spanning sets for the
codes \(\mathcal S_{q,m,m}\) in the corresponding cases.

\begin{proposition}
With the notation above, the following hold.

\begin{enumerate}
    \item We have \[\operatorname{strk}(\mathcal S_{q,2,2})=3.\] A symmetric rank-one spanning set is given by
    \[
    \left\{
    \operatorname{Tr}_{q^2/q}(x),\;
    \eta \operatorname{Tr}_{q^2/q}(\eta x),\;
    \eta^2 \operatorname{Tr}_{q^2/q}(\eta^2 x)
    \right\},
    \]
    where \(\eta\in \mathbb F_{q^2}^\ast\) satisfies $\eta^{2q}-\eta^{2q-1}-\eta^{q+1}+\eta^{q-1}+\eta-1\neq 0$.

    \item For every prime power \(q\), we have
    \[
    \operatorname{strk}(\mathcal S_{q,3,3})\le 6.
    \]
    In particular,
    \[
    \operatorname{strk}(\mathcal S_{2,3,3})
    =
    \operatorname{strk}(\mathcal S_{3,3,3})
    =
    6.
    \]
    Moreover, if \(q\ge 4\), then
    \[
    \operatorname{strk}(\mathcal S_{q,3,3})=5.
    \]
    A symmetric rank-one spanning set with six elements is given by
    \[
    \left\{
    \operatorname{Tr}_{q^3/q}(x),\;
    \xi\operatorname{Tr}_{q^3/q}(\xi x),\;
    \ldots,\;
    \xi^5\operatorname{Tr}_{q^3/q}(\xi^5 x)
    \right\},
    \]
    where \(\xi\in \mathbb F_{q^3}\) is chosen as in
    Corollary~\ref{cor:m=3}.

    \item We have
    \[
    \operatorname{strk}(\mathcal S_{2,4,4})=9.
    \]
    A symmetric rank-one spanning set is given by
    \[
    \left\{
    \xi^{i_1}\operatorname{Tr}_{2^4/2}(\xi^{i_1}x),\ldots,
    \xi^{i_9}\operatorname{Tr}_{2^4/2}(\xi^{i_9}x)
    \right\},
    \]
    where \(\xi\) and \(i_1,\ldots,i_9\) are as in Table~\ref{table:m4}.

    \item We have
    \[
    \operatorname{strk}(\mathcal S_{3,4,4})\le 9.
    \]
    A symmetric rank-one spanning set is given by
    \[
    \left\{
    \xi^{i_1}\operatorname{Tr}_{3^4/3}(\xi^{i_1}x),\ldots,
    \xi^{i_9}\operatorname{Tr}_{3^4/3}(\xi^{i_9}x)
    \right\},
    \]
    where \(\xi\) and \(i_1,\ldots,i_9\) are as in Table~\ref{table:m4}.

    \item For \(q\in\{4,5\}\), we have
    \[
    \operatorname{strk}(\mathcal S_{q,4,4})=8.
    \]
    A symmetric rank-one spanning set is given by
    \[
    \left\{
    \xi^{i_1}\operatorname{Tr}_{q^4/q}(\xi^{i_1}x),\ldots,
    \xi^{i_8}\operatorname{Tr}_{q^4/q}(\xi^{i_8}x)
    \right\},
    \]
    where \(\xi\) and \(i_1,\ldots,i_8\) are as in Table~\ref{table:m4}.
\end{enumerate}
\end{proposition}

We now give an explicit matrix realization of one of the decompositions above. The previous results are stated in terms of rank-one symmetric linearized
polynomials. Via the correspondence among symmetric linearized polynomials,
symmetric bilinear forms, and symmetric matrices, these decompositions can be
translated into decompositions inside \(\Sym_q(m)\). The following example
illustrates this translation explicitly.

\begin{example}
    Consider the symmetric rank-metric code $C\leq\Sym_2(4)$ spanned by the following matrices.

    \begin{equation*}
        C_1=\begin{pmatrix}
1 & 0 & 0 & 1\\
0 & 0 & 1 & 0\\
0 & 1 & 0 & 0\\
1 & 0 & 0 & 0\\
\end{pmatrix},\;
C_2=\begin{pmatrix}
1 & 0 & 0 & 0\\
0 & 0 & 0 & 1\\
0 & 0 & 1 & 0\\
0 & 1 & 0 & 0\\
\end{pmatrix},\;
C_3=\begin{pmatrix}
0 & 1 & 0 & 0\\
1 & 1 & 0 & 0\\
0 & 0 & 0 & 1\\
0 & 0 & 1 & 0\\
\end{pmatrix},\;
C_4=\begin{pmatrix}
0 & 0 & 1 & 0\\
0 & 1 & 1 & 0\\
1 & 1 & 0 & 0\\
0 & 0 & 0 & 1\\
\end{pmatrix}.
    \end{equation*}
One can  check that this is the one-dimensional Gabidulin code \(\langle x\rangle_{\mathbb F_{2^4}}\),
represented with respect to the ordered basis $(1,a,a^2,a^3)$ and its trace-dual basis $(a^{14},a^2,a,1)$. Using the same bases, the rank-one symmetric linearized polynomials determined
by the exponents in Table~\ref{table:m4} for \(q=2\) are represented by the
following matrices.
\begin{equation*}
\begin{array}{ccc}
A_1=\begin{pmatrix}
1 & 0 & 0 & 0\\
0 & 0 & 0 & 0\\
0 & 0 & 0 & 0\\
0 & 0 & 0 & 0
\end{pmatrix},&
A_2=\begin{pmatrix}
0 & 0 & 0 & 0\\
0 & 1 & 0 & 0\\
0 & 0 & 0 & 0\\
0 & 0 & 0 & 0
\end{pmatrix},&
A_3=\begin{pmatrix}
1 & 1 & 0 & 0\\
1 & 1 & 0 & 0\\
0 & 0 & 0 & 0\\
0 & 0 & 0 & 0
\end{pmatrix},\\[6ex]
A_4=\begin{pmatrix}
0 & 0 & 0 & 0\\
0 & 1 & 1 & 0\\
0 & 1 & 1 & 0\\
0 & 0 & 0 & 0
\end{pmatrix},&
A_5=\begin{pmatrix}
0 & 0 & 0 & 0\\
0 & 0 & 0 & 0\\
0 & 0 & 1 & 1\\
0 & 0 & 1 & 1
\end{pmatrix},&
A_6=\begin{pmatrix}
0 & 0 & 0 & 0\\
0 & 1 & 0 & 1\\
0 & 0 & 0 & 0\\
0 & 1 & 0 & 1
\end{pmatrix},\\[6ex]
A_7=\begin{pmatrix}
1 & 1 & 1 & 0\\
1 & 1 & 1 & 0\\
1 & 1 & 1 & 0\\
0 & 0 & 0 & 0
\end{pmatrix},&
A_8=\begin{pmatrix}
0 & 0 & 0 & 0\\
0 & 1 & 1 & 1\\
0 & 1 & 1 & 1\\
0 & 1 & 1 & 1
\end{pmatrix},&
A_9=\begin{pmatrix}
1 & 0 & 0 & 1\\
0 & 0 & 0 & 0\\
0 & 0 & 0 & 0\\
1 & 0 & 0 & 1
\end{pmatrix}.
\end{array}
\end{equation*}
One can check that \(C\) is contained in the \(\mathbb F_2\)-span of these
rank-one symmetric matrices. Indeed, we have 
\begin{align*}
    C_1&=A_5+A_6+A_8+A_9,\\
    C_2&=A_1+A_4+A_5+A_8,\\
    C_3&=A_1+A_2+A_3+A_4+A_6+A_8,\\
    C_4&=A_3+A_4+A_5+A_6+A_7+A_8.
\end{align*}
Thus \(\{A_1,\ldots,A_9\}\) gives an explicit symmetric rank-one spanning set
for \(C\).
\end{example}

\section{Final remarks}

In this paper, we developed a linearized-polynomial approach to the
symmetric tensor decomposition of the multiplication map
\[
M_{q^m}:\mathbb F_{q^m}\times \mathbb F_{q^m}\to \mathbb F_{q^m}.
\]
Using the correspondence among symmetric linearized polynomials,
symmetric bilinear forms, and symmetric matrices, we translated the
existence of symmetric tensor decompositions into explicit trace
representations and linear systems over finite fields. This provides a
constructive framework for searching for decompositions and for verifying
them directly. In small extension degrees, the method recovers known
values of \(\mu_q^{\mathrm{sym}}(m)\) and produces explicit symmetric
rank-one spanning sets.

The point of view developed here also places the problem naturally in the
setting of symmetric rank-metric codes. Indeed, the first-slice space of
the multiplication tensor is equivalent to the one-dimensional Gabidulin
code
\[
\mathcal{G}_1(m,q)=\langle x\rangle_{\mathbb F_{q^m}},
\]
which is contained in the space \(S_q(m)\) of symmetric linearized
polynomials. Thus the symmetric tensor rank of \(M_{q^m}\) can be
interpreted as the symmetric tensor rank of this natural one-dimensional
symmetric rank-metric code. This reformulation suggests that techniques
from rank-metric coding theory may be useful not only for deriving lower
bounds, but also for constructing and organizing explicit decompositions.

Several questions remain open. A direction is to determine the exact
values of~\(\mu_q^{\mathrm{sym}}(m)\) for further parameters, especially
for small \(q\), where the general equality \(\mu_q^{\mathrm{sym}}(m)=2m-1\)
does not apply. A related problem is to find uniform families of elements
\(\alpha_1,\ldots,\alpha_R\in \mathbb F_{q^m}\) satisfying the linear
conditions obtained in this paper. Such families would lead to explicit
constructions of symmetric tensor decompositions beyond the small values
of \(m\) treated here.

Another natural direction is to extend the coding-theoretic interpretation
to broader classes of symmetric rank-metric codes. The invariant
\(\operatorname{strk}(C)\) measures how efficiently a symmetric code can be
covered by rank-one symmetric matrices. It would be interesting to
understand how this invariant interacts with classical code parameters,
such as dimension, minimum distance, and MRD properties, and whether these
parameters impose meaningful lower or upper bounds on
\(\operatorname{strk}(C)\).

Finally, the formulation in terms of Frobenius-stable linear systems
suggests algorithmic questions. One may ask for efficient procedures to
search for suitable elements \(\alpha_i\), to certify non-existence for a
given length \(R\), or to exploit the symmetry and Frobenius structure in
order to reduce the computational complexity of the search. Developing
such algorithms could make the framework applicable to larger extension
degrees and could provide new data for the study of symmetric bilinear
complexity over finite fields.

\bigskip
\bibliographystyle{abbrv}
\bibliography{ourbib}

\end{document}